\documentclass[11pt,a4paper]{amsart}
\usepackage[arrow,matrix,curve]{xy}

\title{On derived equivalences of lines, rectangles and triangles}
\author{Sefi Ladkani}
\address{Max-Planck-Institut f\"{u}r Mathematik, Vivatsgasse 7, 53111 Bonn,
Germany}
\curraddr{Mathematical Institute of the University of Bonn, Endenicher
Allee 60, 53115 Bonn, Germany}
\email{sefil@math.uni-bonn.de}

\thanks{This work was supported in part by a European Postdoctoral
Institute fellowship.}

\DeclareMathOperator{\add}{add}
\DeclareMathOperator{\Aus}{Aus}
\DeclareMathOperator{\sAus}{\underline{Aus}}
\DeclareMathOperator{\End}{End}
\DeclareMathOperator{\h}{H}
\DeclareMathOperator{\Hom}{Hom}
\DeclareMathOperator{\Mod}{Mod}
\DeclareMathOperator{\modf}{mod}
\DeclareMathOperator{\per}{per}
\DeclareMathOperator{\Proj}{Proj}
\DeclareMathOperator{\proj}{proj}
\DeclareMathOperator{\RHom}{\mathbf{R}Hom}

\newcommand{\bZ}{\mathbb{Z}}
\newcommand{\cA}{\mathcal{A}}
\newcommand{\cC}{\mathcal{C}}
\newcommand{\cD}{\mathcal{D}}
\newcommand{\cK}{\mathcal{K}}
\newcommand{\cT}{\mathcal{T}}

\newcommand{\eps}{\varepsilon}
\newcommand{\gL}{\Lambda}
\newcommand{\vphi}{\varphi}

\newcommand{\dA}{\cD(A)}
\newcommand{\dB}{\cD(B)}

\newcommand{\dAB}{\cD(A \ten B)}
\newcommand{\dgL}{\cD(\gL)}
\newcommand{\tgL}[1]{\cT_{#1}(\gL)}

\newcommand{\ten}{\otimes}
\newcommand{\tenl}{\stackrel{\mathbf{L}}{\otimes}}

\theoremstyle{plain}
\newtheorem{thm}{Theorem}

\newtheorem*{thm*}{Theorem}

\newtheorem{lemma}{Lemma}[section]
\newtheorem{prop}[lemma]{Proposition}
\newtheorem{cor}[lemma]{Corollary}

\theoremstyle{definition}
\newtheorem{remark}[lemma]{Remark}

\numberwithin{equation}{section}

\begin{document}

\begin{abstract}
We present a method to construct new tilting complexes from existing
ones using tensor products, generalizing a result of Rickard. The
endomorphism rings of these complexes are generalized matrix rings that
are ``componentwise'' tensor products, allowing us to obtain many
derived equivalences that have not been observed by using previous
techniques.

Particular examples include algebras generalizing the ADE-chain related
to singularity theory, incidence algebras of posets and certain
Auslander algebras or more generally endomorphism algebras of initial
preprojective modules over path algebras of quivers. Many of these
algebras are fractionally Calabi-Yau and we explicitly compute their CY
dimensions. Among the quivers of these algebras one can find shapes of
lines, rectangles and triangles.
\end{abstract}

\maketitle

\section*{Introduction}

This work deals with derived equivalences of various rings and
algebras. One could argue that the question of derived equivalence has
been settled by the seminal result of Rickard~\cite{Rickard89}, stating
that for two rings $R$ and $S$, the derived categories of modules
$\cD(R)$ and $\cD(S)$ are equivalent as triangulated categories if and
only if there exists a so-called tilting complex $T \in \cD(R)$
such that $\End_{\cD(R)}(T) \simeq S$.

Nevertheless, when we are faced with two explicit such rings and want
to assess their derived equivalence, it is sometimes notoriously
difficult to decide whether a tilting complex exists, and if so, to
construct it explicitly.

It is therefore mostly beneficial to have at our disposal techniques to
construct tilting complexes. Such techniques have been invented in
relation with tilting theory of finite-dimensional algebras and modular
representation theory, see for example the
books~\cite{HandbookTilting07} and~\cite{KoenigZimmermann98} and the
many references therein.

We present a systematic way to construct new tilting complexes from
existing ones using tensor products, generalizing a result of Rickard
in~\cite{Rickard91}. In our variation, the resulting endomorphism rings
are no longer tensor products as in~\cite{Rickard91}, but rather
generalized matrix rings that are ``componentwise'' tensor products. As
the class of such rings is much broader, this allows us to obtain many
derived equivalences that have not been observed by using previous
techniques.

\begin{figure}[b]
\[
\xymatrix@R=1pc@C=0.5pc{
& & & {\bullet} \ar[dr] \\
& & {\bullet} \ar[ur] \ar@{.}[rr] \ar[dr] & &
{\bullet} \ar[dr] \\
& {\bullet} \ar[ur] \ar@{.}[rr] \ar[dr] &&
{\bullet} \ar[ur] \ar@{.}[rr] \ar[dr] &&
{\bullet} \ar[dr] \\
{\bullet} \ar[ur] \ar@{.}[rr] & &
{\bullet} \ar[ur] \ar@{.}[rr] & &
{\bullet} \ar[ur] \ar@{.}[rr] & &
{\bullet}
}
\qquad
\xymatrix@R=1pc@C=0.5pc{
{\bullet} \ar[dr] & &  {\bullet} \ar[ll] \ar[dr] \\
& {\bullet} \ar@{.}[ur] \ar[dr] & & {\bullet} \ar[ll] \ar[dr] \\
& & {\bullet} \ar@{.}[ur] \ar[dr] & & {\bullet} \ar[ll] \ar[dr] \\
& & & {\bullet} \ar@{.}[ur] \ar[dr] & & {\bullet} \ar[ll] \ar[dr] \\
& & & & {\bullet} \ar@{.}[ur] & & {\bullet} \ar[ll]
}
\]
\[
\xymatrix@=1.5pc{
{\bullet} \ar[r] \ar@{.}@/^0.75pc/[rrr] &
{\bullet} \ar[r] \ar@{.}@/_0.75pc/[rrr] &
{\bullet} \ar[r] \ar@{.}@/^0.75pc/[rrr] &
{\bullet} \ar[r] \ar@{.}@/_0.75pc/[rrr] &
{\bullet} \ar[r] \ar@{.}@/^0.75pc/[rrr] &
{\bullet} \ar[r] \ar@{.}@/_0.75pc/[rrr] &
{\bullet} \ar[r] \ar@{.}@/^0.75pc/[rrr] &
{\bullet} \ar[r] &
{\bullet} \ar[r] &
{\bullet}
}
\]
\caption{Quivers with relations of three derived equivalent algebras:
$A(10,3)$ (\emph{line}), the tensor product $kA_5 \ten_k kA_2$
(\emph{rectangle}) and the stable Auslander algebra of the path algebra
of $kA_5$ (\emph{triangle}).} \label{fig:trirectline}
\end{figure}
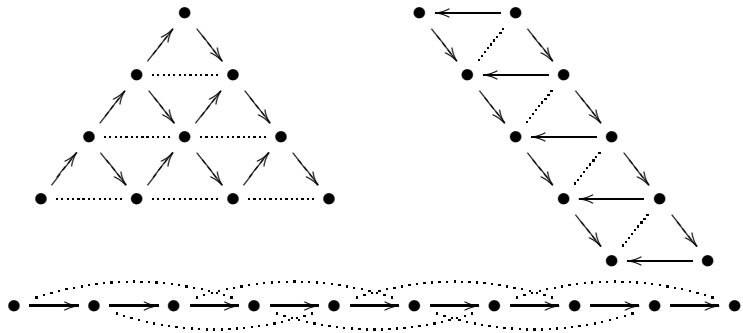

As an example of our methods, we consider the following three families
of algebras. The first, visualized pictorially as \emph{``lines''},
consists of algebras denoted $A(n,r+1)$ which arise from the linear
quiver $A_n$ by taking its path algebra (over some field $k$) modulo
the ideal generated by all the paths of a given length $r+1$. The
second, consisting of the tensor products (over $k$) $kA_r \ten_k kA_n$
with $r, n \geq 1$, can be visualized as \emph{``rectangles''}, as
their quivers are rectangles (with fully commutative relations).
Members of these two families have been considered
in~\cite{LenzingdelaPena08} in connection with derived accessible
algebras and spectral methods.

It will be a consequence of our results that for any value of $r$ and
$n$, the line $A(r \cdot n, r+1)$ and the rectangle $kA_r \ten_k kA_n$
are derived equivalent. As, for $r=2$ and $n \leq 8$, each algebra
$A(n,3)$ is derived equivalent to the corresponding path algebra in the
sequence $A_1, A_2, A_3, D_4, D_5, E_6, E_7, E_8$ known as the ADE
chain~\cite{ADE08}, one may think of these algebras for $r > 2$ as
higher ADE-chains.

The relevance of these algebras to other branches of representation
theory is demonstrated by the result of Kussin, Lenzing and
Meltzer~\cite{Lenzing10}, who have recently shown a relation between
the ``lines'' and the categories of coherent sheaves on weighted
projective lines, through the notion of the stable category of vector
bundles~\cite{LenzingdelaPena08}. In addition, as shown in the same
work, and also recently in~\cite{Chen09}, the stable categories of
submodules of nilpotent linear maps studied by Ringel and
Schmidmeier~\cite{RingelSchmidmeier08} are equivalent to bounded
derived categories of certain ``rectangles''.

The third family of algebras that we consider arises as the
endomorphism rings of certain initial modules in the preprojective
component of path algebras of quivers $Q$ without oriented cycles, in a
dual fashion to the terminal modules introduced by Geiss, Leclerc and
Schr\"{o}er~\cite[\S 2]{GLS08}, and includes also many Auslander
algebras of Dynkin quivers and their stable analogues. It will be again
a consequence of our results that such algebras are derived equivalent
to the tensor product of the path algebra of $Q$ with the path algebra
$kA_r$ for a suitable value of $r$. As the quivers of some of these
algebras have a shape of \emph{``triangles''} when $Q$ is the linearly
oriented diagram $A_n$, we obtain a derived equivalence between them
and the ``rectangles''.

One of the simplest examples of derived equivalence between members of
these three families is shown in Figure~\ref{fig:trirectline},
corresponding to the case where $n=5$ and $r=2$.

The paper is organized as follows. In Section~\ref{sec:results} we
present, in detail, our results. The proof of our main
Theorem~\ref{t:tensor} is given in Section~\ref{sec:prooftensor}. The
proofs of its various consequences are the contents of
Section~\ref{sec:proofother}.

\section{The results}
\label{sec:results}

\subsection{Tilting complexes for tensor products}

Let $A$ be an algebra over a commutative ring $k$. As the notion of
a tilting complex is central to our investigations, we recall that a
complex $T \in \dA$ is a \emph{tilting complex} if it has the following
two properties:
\begin{itemize}
\item
$T$ is \emph{exceptional}, that is, $\Hom_{\dA}(T, T[r])=0$ for all
$r \neq 0$;
\item
$T$ is a \emph{compact generator}, i.e.\ the smallest triangulated
subcategory of $\dA$ containing $T$ and closed under isomorphisms and
direct summands equals the subcategory of complexes (quasi-)isomorphic
to bounded complexes of finitely generated projective $A$-modules.
\end{itemize}

Let $B$ be an algebra over $k$, which is projective as a $k$-module.
Fix a tilting complex $U$ of projective
$B$-modules whose endomorphism algebra is projective as
$k$-module. Then for any tilting complex $T$ of $A$-modules, a theorem
of Rickard~\cite[Theorem~2.1]{Rickard91} tells us that $T \ten_k U$ is
a tilting complex for the tensor product $A \ten_k B$, with
endomorphism algebra which is again a tensor product, namely
$\End_{\dA} T \otimes_k \End_{\dB} U$.

Assume that $U$ decomposes as $U = U_1 \oplus U_2 \oplus \dots \oplus U_n$
(the $U_i$ need not be indecomposable). Obviously,
\[
T \ten_k U = (T \ten_k U_1) \oplus (T \ten_k U_2) \oplus \dots \oplus
(T \ten_k U_n) .
\]

Consider now a variation, where instead of taking just one tilting
complex $T$, we take $n$ of them, say $T_1, T_2, \dots, T_n$, and
replace each summand $T \ten_k U_i$ by $T_i \ten_k U_i$. Our following
Theorem~\ref{t:tensor} gives conditions on the $T_i$ which guarantee
that the resulting complex is tilting, and moreover computes its
endomorphism algebra in terms of those of $T_i$ and $U_i$.

\begin{thm}
\label{t:tensor} Let $k$ be a commutative ring and let $A$ and $B$ be
two $k$-algebras, with $B$ projective as $k$-module. Let $U_1, U_2, \dots,
U_n \in \dB$ be complexes bounded from above satisfying:
\begin{enumerate}
\renewcommand{\theenumi}{\roman{enumi}}
\item
$U = U_1 \oplus U_2 \oplus \dots \oplus U_n$ is a tilting complex in
$\dB$;
\item
\label{i:Uproj} The terms $U^r$ of the complex $U=(U^r)$ are projective
as $k$-modules;
\item
\label{i:EndUproj}
The endomorphism $k$-algebra $\End_{\dB}(U)$ is projective as
$k$-module.
\end{enumerate}

Let $T_1, T_2, \dots, T_n \in \dA$ be tilting complexes with the
property that for any $1 \leq i, j \leq n$,
\[
\Hom_{\dA}(T_i, T_j[r]) = 0 \text{ for any $r \neq 0$, whenever }
\Hom_{\dB}(U_i, U_j) \neq 0 .
\]

Then the complex
\begin{equation} \label{e:TUcomplex}
(T_1 \ten_k U_1) \oplus (T_2 \ten_k U_2) \oplus \dots \oplus
(T_n \ten_k U_n)
\end{equation}
of $(A \ten_k B)$-modules is a tilting complex in $\cD(A \ten_k B)$,
and its endomorphism ring is given by the matrix algebra
\begin{equation} \label{e:endo}
\begin{pmatrix}
M_{11} & M_{12} & \dots & M_{1n} \\
M_{21} & M_{22} & \dots & M_{2n} \\
\vdots & \vdots & \ddots & \vdots \\
M_{n1} & M_{n2} & \dots & M_{nn}
\end{pmatrix}
\end{equation}
where $M_{ij} = \Hom_{\dA}(T_j, T_i) \otimes_k \Hom_{\dB}(U_j, U_i)$
and the multiplication maps $M_{ij} \ten_k M_{jl} \to M_{il}$ are given
by the obvious compositions.
\end{thm}

Note that in general, the matrix ring~\eqref{e:endo} is not a tensor
product of two algebras, but rather a ``componentwise'' tensor product.
Namely, the $(i,j)$-th entry of~\eqref{e:endo} is the tensor product of
the corresponding $(i,j)$-th entries of the rings $\End_{\dA}(T_1
\oplus \dots \oplus T_n)$ and $\End_{\dB}(U) = \End_{\dB}(U_1 \oplus
\dots \oplus U_n)$, both viewed as $n$-by-$n$ matrix rings whose
$(i,j)$-th entry is $\Hom_{\dA}(T_j, T_i)$ and $\Hom_{\dB}(U_j, U_i)$,
respectively. Hence the theorem can produce many algebras which on
first sight might look far from being a tensor product, but
nevertheless such a product structure is apparently hidden in their
derived category.

The purpose of the technical conditions~(\ref{i:Uproj})
and~(\ref{i:EndUproj}) is just to ensure that the functors $- \ten_k
U_i$ are well-behaved. For example, when $k$ is a field, these
conditions are automatically satisfied. They are also satisfied in our
main applications, where the ring $B$, the terms of the complex $U$ and
the endomorphism ring $\End_{\dB}(U)$ are finitely generated and free
as $\bZ$-modules.

Theorem~\ref{t:tensor} provides us with a machinery to produce many new
derived equivalences. Namely, by fixing $B$ and the $U_i$ (usually of
combinatorial origin), we obtain for any $k$-algebra $A$ and tilting
complexes $T_1, \dots, T_n$ over $A$ satisfying some compatibility
conditions, a derived equivalence between $A \ten_k B$ and an algebra
built from the $T_i$ in a prescribed manner depending only on the
$U_i$. This will be demonstrated in the applications, see
Sections~\ref{ssec:ade} and~\ref{ssec:endo} below.

\subsection{Truncated algebras and ADE chains}
\label{ssec:ade}

For our first application, recall that the quiver $A_n$ ($n \geq 1$) is
the following directed graph on $n$ vertices
\[
\xymatrix{
{\bullet_1} \ar[r] & {\bullet_2} \ar[r] & \dots \ar[r] & {\bullet_n}
} ,
\]
and the path algebra $kA_n$ can be viewed as the $k$-algebra of upper
triangular $n$-by-$n$ matrices with entries in $k$. For any $k$-algebra
$\gL$, the $k$-algebra $\gL \ten_k kA_n$ thus consists of upper
triangular $n$-by-$n$ matrices with entries in $\gL$, and we denote it
by $\cT_n(\gL)$.

\begin{thm} \label{t:ade}
Let $\gL$ be a ring and let $T_1, \dots, T_n$ be tilting complexes in $\dgL$
satisfying $\Hom_{\dgL}(T_i, T_{i+1}[r]) = 0$ for all $1 \leq i < n$ and
$r \neq 0$.

Then the following matrix ring, with $\Hom$ and $\End$ computed in $\dgL$,
\[
\begin{pmatrix}
\End T_1 & 0 & 0 & \ldots & 0 \\
\Hom (T_1, T_2) & \End T_2 & 0 & \ldots & \vdots \\
0 & \Hom (T_2, T_3) & \End T_3 & \ddots & \vdots \\
\vdots & \ddots & \ddots & \ddots & 0 \\
0 & \ldots & 0 & \Hom (T_{n-1}, T_n) & \End T_n
\end{pmatrix}
\]
is derived equivalent to $\tgL{n}$.
\end{thm}

Recall that a $\gL$-module is a \emph{tilting module} if, when viewed
as a complex (concentrated in degree $0$), it is a tilting complex. The
above theorem can be reformulated in the language of iterated tilting,
as follows.

\begin{cor} \label{c:itertilt}
Let $\gL$ be a ring.
Set $\gL_1 = \gL$ and define, for $1 \leq i < n$, $\gL_{i+1} =
\End_{\gL_i}(Q_i)$ where $Q_i$ is a tilting $\gL_i$-module.

Then the matrix ring
\begin{equation} \label{e:itertilt}
\begin{pmatrix}
\gL_1 & 0 & 0 & \ldots & 0 \\
Q_1 & \gL_2 & 0 & \ldots & \vdots \\
0 & Q_2 & \gL_3 & \ddots & \vdots \\
\vdots & \ddots & \ddots & \ddots & 0 \\
0 & \ldots & 0 & Q_{n-1} & \gL_n
\end{pmatrix}
\end{equation}
is derived equivalent to $\tgL{n}$.
\end{cor}

As an application, we consider the following two families of algebras
over a commutative ring $k$. The first, consists of algebras denoted
$A(n, m+1)$ that are the quotient of the path algebra of $A_n$ by the
ideal generated by all the paths of length $m+1$. The piecewise
hereditary such algebras (when $k$ is an algebraically closed field)
have been investigated in~\cite{HappelSeidel10}.

In the spirit of the philosophy of~\cite{LenzingdelaPena08}, when
studying the derived equivalence class of the algebras $A(n,m+1)$, one
should not look at a single such algebra each time, but rather consider
them in a sequence, that is, fix $m$ and let the number of vertices $n$
vary. For $m=1$, the algebra $A(n,2)$ is derived equivalent to the path
algebra of $A_n$ (without relations), as follows from~\cite[(IV,
6.7)]{Happel88}. Fixing $m=2$ and setting $n=1,2,\dots,8$, the algebras
$A(n,3)$ are derived equivalent to the path algebras of the quivers in
the sequence
\[
A_1, A_2, A_3, D_4, D_5, E_6, E_7, E_8
\]
known as the \emph{ADE-chain}~\cite{ADE08}, as observed
in~\cite{LenzingdelaPena08} using spectral techniques.

The ADE-chain occurs also in our second family, which consists of the
algebras $kA_m \ten_k kA_n$. These can be viewed as the incidence
algebras of the products $A_m \times A_n$ of two linear orders, or more
pictorially as \emph{``rectangles''}, since they can be identified with
the path algebras of the fully commutative rectangle with $m \times n$
vertices, see Figure~\ref{fig:rectline}. When $m=1$, we obviously get
back the path algebra $kA_n$. But for $m=2$, setting $n=1, 2, 3, 4$,
the algebras $kA_2 \ten_k kA_n$ are derived equivalent to the path
algebras of the quivers in the sequence $A_2, D_4, E_6, E_8$, so that
$kA_2 \ten_k kA_n$ is derived equivalent to $A(2n, 3)$ for $n \leq 4$.
The following result generalizes these equivalences and puts them into
perspective.

\begin{cor} \label{c:rectline}
Let $k$ be a commutative ring and let $m,n \geq 1$. Then the
$k$-algebras $A(m \cdot n, m+1)$ and $kA_m \ten_k kA_n$ are derived
equivalent. In particular, $A(m \cdot n, m+1)$ and $A(m \cdot n, n+1)$
are derived equivalent.
\end{cor}

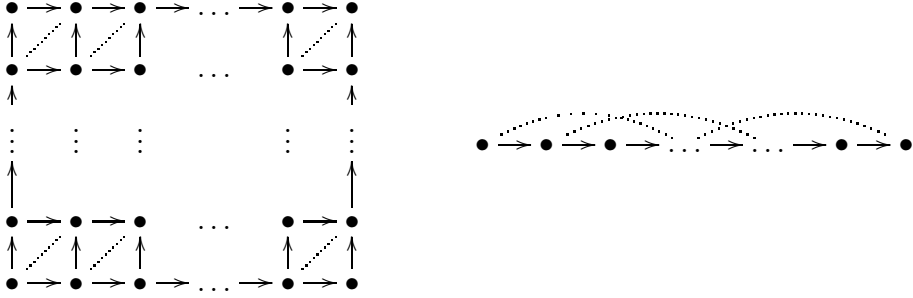
\begin{figure}
\begin{align*}
\begin{array}{c}
\xymatrix@=1pc{
{\bullet} \ar[r] & {\bullet} \ar[r] & {\bullet} \ar[r] &
{\ldots} \ar[r] & {\bullet} \ar[r] & {\bullet} \\
{\bullet} \ar[u] \ar[r] \ar@{.}[ur] & {\bullet} \ar[u] \ar[r] \ar@{.}[ur] &
{\bullet} \ar[u] & {\ldots} & {\bullet} \ar[u] \ar[r] \ar@{.}[ur] &
{\bullet} \ar[u] \\
{\vdots} \ar[u] & {\vdots} & {\vdots} & & {\vdots} & {\vdots} \ar[u] \\
{\bullet} \ar[u] \ar[r] & {\bullet} \ar[r] & {\bullet} & {\ldots} &
{\bullet} \ar[r] & {\bullet} \ar[u] \\
{\bullet} \ar[u] \ar[r] \ar@{.}[ur] & {\bullet} \ar[u] \ar[r] \ar@{.}[ur] &
{\bullet} \ar[u] \ar[r] & {\ldots} \ar[r] & {\bullet} \ar[u] \ar[r] \ar@{.}[ur] &
{\bullet} \ar[u]
}
\end{array}
& \qquad &
\xymatrix@=1pc{
{\bullet} \ar[r] \ar@{.}@/^1pc/[rrr] & {\bullet} \ar[r] \ar@{.}@/^1pc/[rrr] &
{\bullet} \ar[r] & {\ldots} \ar[r] \ar@{.}@/^1pc/[rrr] & {\ldots} \ar[r] &
{\bullet} \ar[r] & {\bullet}
}
\end{align*}
\caption{Quivers with relations of the ``rectangle'' (left) and the
``line'' (right) algebras. The dotted lines indicate the relations.}
\label{fig:rectline}
\end{figure}

Furthermore, as the derived equivalence actually holds over arbitrary
commutative ring $k$, it can be considered as ``universal'', depending
only on the underlying combinatorics and not on the algebraic data.

\smallskip

For the rest of this subsection, assume that $k$ is a field, and denote
by $D = \Hom_k(-, k)$ the usual duality. If $\gL$ is a finite
dimensional algebra over $k$, then $D\gL_{\gL}$ is an injective
co-generator. When it has finite projective dimension and $\gL_{\gL}$
has finite injective dimension, the algebra $\gL$ is called
\emph{Gorenstein}.

\begin{cor} \label{c:replica}
Let $k$ be a field and let $\gL$ be a finite dimensional $k$-algebra
which is Gorenstein. Then the triangular matrix algebras $\tgL{n}$ and
\[
\begin{pmatrix}
\gL & D\gL & 0 & \ldots & 0 \\
0 & \gL & D\gL & \ddots & \vdots \\
\vdots & 0 & \gL & \ddots & 0 \\
\vdots & & \ddots & \ddots & D\gL \\
0 & \ldots & \ldots & 0 & \gL
\end{pmatrix}
\]
are derived equivalent.
\end{cor}

The algebra appearing in the corollary is known as the
$(n-1)$-\emph{replicated algebra} of $\gL$ and it has connections with
$(n-1)$-cluster-categories, see~\cite{ABST08}. It is not clear a-priori
why replicated algebras of two derived equivalent Gorenstein algebras
should also be derived equivalent. But, in view of
Corollary~\ref{c:replica}, as tensor products behave well with respect
to derived equivalences, this is also the case for replicated algebras:
\begin{cor}
The $n$-replicated algebras of two derived equivalent, finite
dimensional, Gorenstein algebras are also derived equivalent.
\end{cor}

\subsection{Endomorphism algebras}
\label{ssec:endo}

Another application of Theorem~\ref{t:tensor} concerns the endomorphism
algebras of direct sums of tilting complexes. We start with the
following result on ``configurations'' of tilting complexes according
to the poset $\{ 1 < 2 < \dots < n\}$. It is possible to formulate
analogous results for general posets, but such generality is not needed
for our purposes.

\begin{thm} \label{t:endo}
Let $\gL$ be a ring, and let $T_1, \dots, T_n$ be tilting complexes in
$\dgL$ satisfying $\Hom_{\dgL}(T_i, T_j[r]) = 0$ for all $1 \leq i < j
\leq n$ and $r \neq 0$.

Then the matrix ring (with $\Hom$ and $\End$ computed in $\dgL$)
\[
\begin{pmatrix}
\End T_1 & 0 & 0 & \ldots & 0 \\
\Hom (T_1, T_2) & \End T_2 & 0 & \ldots & 0 \\
\Hom (T_1, T_3) & \Hom (T_2, T_3) & \End T_3 & \ldots & 0 \\
\vdots & \vdots & \vdots & \ddots & \vdots \\
\Hom (T_1, T_n) & \Hom (T_2, T_n) & \Hom (T_3, T_n) & \dots & \End T_n
\end{pmatrix}
\]
is derived equivalent to $\tgL{n}$.
\end{thm}

\begin{remark}
When $n=2$, Theorems~\ref{t:ade} and~\ref{t:endo} coincide, and in view
of the equivalent Corollary~\ref{c:itertilt}, they state that when
$\gL$ is a ring and $T$ is a tilting $\gL$-module, the two rings
\[
\begin{pmatrix}
\gL & \gL \\ 0 & \gL
\end{pmatrix}
\quad \text{and} \quad
\begin{pmatrix}
\gL & 0 \\ {_{\End T}T_{\gL}} & \End T
\end{pmatrix}
\]
are derived equivalent. This is a special case of
Theorem~4.5 in~\cite{Ladkani08}.
\end{remark}

If $T$ is a tilting complex for a ring $\gL$, then its endomorphism ring
is derived equivalent to $\gL = \tgL{1}$. The next corollary shows that
more generally, under certain compatibility conditions,
if $T_1, \dots, T_n$ are tilting complexes, then the endomorphism ring of
their sum is derived equivalent to $\tgL{n}$.

\begin{cor} \label{c:endoSum}
Let $\gL$ be a ring and
let $T_1, \dots, T_n$ be tilting complexes in $\dgL$ such that:
\begin{enumerate}
\renewcommand{\theenumi}{\roman{enumi}}
\item
$\Hom_{\dgL}(T_i, T_j[r]) = 0$ for all $i<j$ and $r \neq 0$,
\item
\label{i:endoSum}
$\Hom_{\dgL}(T_j, T_i) = 0$ for all $i<j$.
\end{enumerate}
Then $\End_{\dgL}(T_1 \oplus \dots \oplus T_n)$ and $\tgL{n}$
are derived equivalent.
\end{cor}

We deduce the following result concerning endomorphism algebras arising
from an auto-equivalence of the derived category. Denote by $\per \gL$
the full subcategory of perfect complexes in $\dgL$, which consists of
all the complexes that are isomorphic in $\dgL$ to bounded ones with
finitely generated projective terms.

\begin{cor} \label{c:endoF}
Let $\gL$ be a ring and $F : \per \gL \xrightarrow{\sim} \per \gL$ an
auto-equivalence. Let $e_1 < e_2 < \dots < e_n$ be an increasing
sequence of integers and denote by $\Delta = \left\{ e_j - e_i \,:\, 1
\leq i < j \leq n \right\}$ the set of its (positive) differences.
Assume that for any $d \in \Delta$,
\begin{enumerate}
\renewcommand{\theenumi}{\roman{enumi}}
\item
\label{i:endoFnz} $\h^r(F^d \gL) = 0$ for all $r \neq 0$;

\item
\label{i:endoFz} $\h^0(F^{-d} \gL) = 0$.
\end{enumerate}
Then $\End_{\dgL}(F^{e_1}\gL \oplus F^{e_2}\gL \oplus \dots \oplus
F^{e_n}\gL)$ is derived equivalent to $\tgL{n}$.
\end{cor}

In particular, for any sequence of $n$ consecutive integers, the
conditions~(\ref{i:endoFnz}) and~(\ref{i:endoFz}) need to be checked
for $d = 0, 1, \dots, n-1$.

\subsection{Fractionally Calabi-Yau algebras}

Throughout this section, $k$ denotes a field. Let $\cT$ be a
triangulated $k$-category with finite dimensional $\Hom$-sets. Recall
that a \emph{Serre functor} on $\cT$ is an auto-equivalence $\nu : \cT
\to \cT$ with bifunctorial isomorphisms
\[
\Hom_{\cT}(X, Y) \xrightarrow{\simeq} D\Hom_{\cT}(Y, \nu X)
\]
for $X, Y \in \cT$, see~\cite{BondalKapranov89}. 

By a \emph{fraction} we mean a pair $(d,e)$ of integers with $e \geq 1$.
By abuse of notation we write a fraction also in the traditional way as
$\frac{d}{e}$, but one should keep in mind that the common factors
cannot always be canceled.
We say that a triangulated $k$-category $\cT$ with shift functor $[1]$
and Serre functor $\nu$ is \emph{fractionally Calabi-Yau of dimension
$\frac{d}{e}$} (or \emph{$\frac{d}{e}$-CY} in short) if
$\nu^e \simeq [d]$, see~\cite{Kontsevich98}. Obviously, being
$\frac{d}{e}$-CY implies being $\frac{\ell d}{\ell e}$-CY for any $\ell
\geq 1$.

For a finite dimensional algebra $\gL$ over $k$, denote by $\modf \gL$
the category of finite dimensional right $\gL$-modules and by
$\cD^b(\modf \gL)$ its bounded derived category. We say that $\gL$ is
\emph{fractionally CY} if $\cD^b(\modf \gL)$ is. An interesting class
of such algebras is provided by~\cite[Theorem~4.1]{MiyachiYekutieli01},
namely the path algebra of any Dynkin quiver is $\frac{h-2}{h}$-CY,
where $h$ is the Coxeter number of the corresponding Dynkin diagram. In
particular, $kA_n$ is $\frac{n-1}{n+1}$-CY. More generally, a
connection between the fractionally CY property and $n$-representation
finiteness is outlined in the recent paper~\cite{HerschendIyama09}.

Fractionally CY algebras behave well with respect to tensor products.
Indeed, define the sum of two fractions $\frac{d_1}{e_1}$ and
$\frac{d_2}{e_2}$ as the fraction $\frac{d}{e}$, where $e$ is the least
common multiple of $e_1, e_2$ and $d$ is set such that $\frac{d}{e} =
\frac{d_1}{e_1} + \frac{d_2}{e_2}$ as rational numbers.
Now, if $A$ is $\alpha$-CY and $B$ is $\beta$-CY, then $A \ten_k B$
is $(\alpha+\beta)$-CY, provided that it has finite global dimension.
This happens, for example, when the field $k$ is perfect, or when $k$
is arbitrary and $B = kA_n$. Thus, starting with two such algebras,
Theorem~\ref{t:tensor} can be used to construct many new fractionally
CY algebras, all of CY-dimension $\alpha+\beta$, which are not
necessarily tensor products of algebras. A particular case is the
following:

\begin{cor} \label{c:fracCY}
Let $\gL$ be $\lambda$-CY. Then $\tgL{n}$ is
$(\lambda+\frac{n-1}{n+1})$-CY, hence the algebras in
Theorem~\ref{t:ade}, Corollaries~\ref{c:itertilt}
(and in particular~\ref{c:replica}),
\ref{c:endoSum} and~\ref{c:endoF} are all $(\lambda + \frac{n-1}{n+1})$-CY.
\end{cor}

The next corollary is an immediate consequence of
Corollary~\ref{c:rectline}. For $m=2$, it provides another explanation
for the CY-dimensions computed in~\cite{Lenzing10}.

\begin{cor}
The line $A(n \cdot m, m+1)$ is $(\frac{m-1}{m+1} + \frac{n-1}{n+1})$-CY.
\end{cor}

\subsection{Path algebras of quivers and Auslander algebras}

The results of Section~\ref{ssec:endo} can be applied in particular for
preprojective components of path algebras of quivers, a setting which
we now recall, see e.g.~\cite{Ringel80}.
Let $k$ denote an algebraically closed field throughout
this section. Let $Q$ be a (finite) quiver without oriented cycles with
set of vertices $Q_0$, and denote by $kQ$ the path algebra of $Q$. The
indecomposable projectives of $kQ$ are in one-to-one correspondence with
the vertices $x \in Q_0$, and we denote them by $\{P_x\}_{x \in Q_0}$.

The category $\cD^b(\modf kQ)$ admits a Serre functor $\nu$ given by
$\nu = - \tenl_{kQ} D(kQ)$, see~\cite{Happel88}. The Auslander-Reiten
translation $\tau$ on $\modf kQ$ can be written as
$\tau = H^0 \circ \nu[-1]$, giving a bijection between the
non-projective and non-injective indecomposables of $\modf kQ$, with inverse
$\tau^- = H^0 \circ \nu^-[1]$ where $\nu^-$ is the inverse of $\nu$.

\begin{cor} \label{c:preproj}
Let $r \geq 0$ such that $\tau^{-r} P_x \neq 0$ for all $x \in Q_0$.
Then the algebras
\[
\End_{kQ}\left(kQ \oplus \tau^{-1} kQ \oplus \dots \oplus \tau^{-r} kQ \right)
\quad \text{and} \quad
kQ \ten_k kA_{r+1} = \cT_{r+1}(kQ)
\]
are derived equivalent.
\end{cor}

For a finite dimensional $k$-algebra $\gL$ with a finite number of
isomorphism classes of indecomposables in $\modf \gL$
(i.e.\ $\gL$ is of \emph{finite representation type}),
the \emph{Auslander algebra} of $\gL$ is defined as
$\Aus(\gL) = \End_{\gL}(M)$ where $M$ is the sum of the (non-isomorphic)
indecomposable $\gL$-modules, see~\cite{Auslander71}.

By Gabriel~\cite{Gabriel72}, the quivers whose path algebra is of
finite representation type are precisely those whose underlying graph
is a Dynkin diagram of type $A$, $D$ or $E$. For such a quiver, let
\[
r_x = \max \left\{ r \geq 0 \,:\, \tau^{-r}P_x \neq 0 \right\}
\]
denote the size of the $\tau^{-}$-orbit of $P_x$ in $\modf kQ$.
Following~\cite{HerschendIyama09}, we call the algebra $kQ$
\emph{homogeneous} if $r_x = r$ does not depend on $x$.

By considering the Auslander-Reiten quiver of $kQ$, as given
in~\cite[\S 6.5]{Gabriel80} or~\cite[\S 2.2]{Ringel80},
it is known that $kQ$ is homogeneous if and only if the orientation is
invariant under the automorphism of the underlying Dynkin diagram as drawn in
Figure~\ref{fig:autDynkin}.
Namely, for $A_n$ and $E_6$ it is the reflection
around the vertical axis at the middle,
for $D_{2n+1}$ it is the reflection around the horizontal axis,
and for the other diagrams it is the identity,
see also~\cite[Prop~3.2]{HerschendIyama09}.
These diagram automorphisms appear also in the theory of finite
twisted groups of Lie type, see e.g.~\cite[\S 4.4]{CarterSegalMacdonald95}.
We call such invariant orientations \emph{symmetric}. Note that there are no
symmetric orientations on the diagram $A_{2n}$ and any orientation on
$D_{2n}$, $E_7$ and $E_8$ is symmetric.

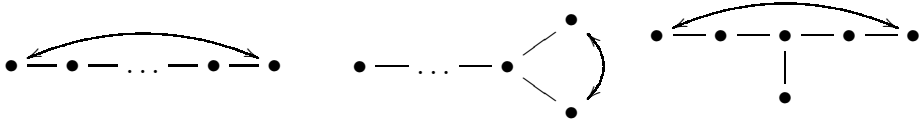
\begin{figure}
\begin{align*}
\begin{array}{c}
\xymatrix@=0.9pc{
{\bullet} \ar@{-}[r] \ar@/^1pc/@{<->}[rrrr]
& {\bullet} \ar@{-}[r] & {\ldots} \ar@{-}[r] & {\bullet} \ar@{-}[r] &
{\bullet}
}
\end{array}
& &
\begin{array}{c}
\xymatrix@C=1pc@R=0.5pc{
& & & {\bullet} \ar@/^1pc/@{<->}[dd] \\
{\bullet} \ar@{-}[r] & {\ldots} \ar@{-}[r] &
{\bullet} \ar@{-}[ur] \ar@{-}[dr] \\
& & & {\bullet}
}
\end{array}
& &
\begin{array}{c}
\xymatrix@=1pc{
{\bullet} \ar@{-}[r] \ar@/^1pc/@{<->}[rrrr] & {\bullet} \ar@{-}[r] &
{\bullet} \ar@{-}[r] \ar@{-}[d] & {\bullet} \ar@{-}[r] & {\bullet} \\
& & {\bullet}
}
\end{array}
\end{align*}
\caption{Automorphisms of the Dynkin diagrams $A_n$, $D_{2n+1}$ and $E_6$.}
\label{fig:autDynkin}
\end{figure}

\begin{cor} \label{c:Aus}
Let $Q$ be a Dynkin quiver with a symmetric orientation.
Then the Auslander algebra $\Aus(kQ)$ is derived equivalent to an
incidence algebra of a poset as indicated in the following table,
\begin{center}
\begin{tabular}{cccr@{$=$}l}
& & Derived type of & \multicolumn{2}{c}{} \\
Diagram & Orientation & Auslander algebra & 
\multicolumn{2}{c}{CY-dimension} \\ \hline
$A_{2n+1}$ & symmetric & $A_{2n+1} \times A_{n+1}$ &
$\frac{2^\eps n(2n+3)}{2^\eps (n+1)(n+2)}$ & 
$\frac{2n}{2n+2} + \frac{n}{n+2}$ \vspace{3pt} \\
$D_{2n}$   & any       & $D_{2n} \times A_{2n-1}$ &
$\frac{(2n-2)(4n-1)}{(2n-1)2n}$ &
$\frac{2n-2}{2n-1} + \frac{2n-2}{2n}$ \vspace{3pt} \\
$D_{2n+1}$ & symmetric & $D_{2n+1} \times A_{2n}$ &
$\frac{(4n-2)(4n+1)}{4n(2n+1)}$ &
$\frac{4n-2}{4n} + \frac{2n-1}{2n+1}$ \vspace{3pt} \\
$E_6$      & symmetric & $E_6 \times A_6$ &
$\frac{130}{84}$ & $\frac{10}{12} + \frac{5}{7}$ \vspace{3pt} \\
$E_7$      & any       & $E_7 \times A_9$ &
$\frac{152}{90}$ & $\frac{8}{9} + \frac{8}{10}$ \vspace{3pt} \\
$E_8$      & any       & $E_8 \times A_{15}$ &
$\frac{434}{240}$ & $\frac{14}{15} + \frac{14}{16}$
\end{tabular}
\end{center}
where in the first row $\eps$ is $0$ when $n$ is even and $1$ otherwise.
In particular, all these Auslander algebras are fractionally Calabi-Yau.
\end{cor}

One should keep in mind that while the derived category $\cD^b(\modf
kQ)$ is independent on the orientation of $Q$, this is no longer true
for the derived category $\cD^b(\modf \Aus(kQ))$. Indeed, even in the
simplest example of the diagram $A_3$, the Auslander algebra
corresponding to the linear orientation is derived equivalent to
$kD_6$, while that of the same diagram but with alternating orientation
(which is symmetric) is derived equivalent to $kE_6$ (or equivalently,
$kA_2 \ten_k kA_3$). The quivers with relations of the two Auslander
algebras are shown in Figure~\ref{fig:ausA3}, where we used white dots
to indicate the vertices corresponding to the indecomposable projective
$kQ$-modules.

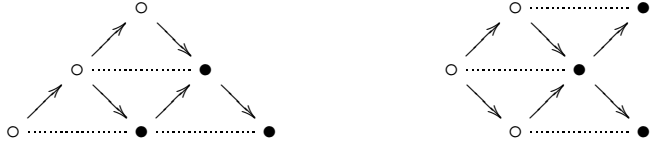
\begin{figure}
\begin{align*}
\xymatrix@=1pc{
& & {\circ} \ar[dr] \\
& {\circ} \ar[ur] \ar[dr] \ar@{.}[rr] & & {\bullet} \ar[dr] \\
{\circ} \ar[ur] \ar@{.}[rr] & & {\bullet} \ar[ur] \ar@{.}[rr]
& & {\bullet}
}
& &
\xymatrix@=1pc{
& {\circ} \ar[dr] \ar@{.}[rr] & & {\bullet} \\
{\circ} \ar[ur] \ar[dr] \ar@{.}[rr] & & {\bullet} \ar[ur] \ar[dr] \\
& {\circ} \ar[ur] \ar@{.}[rr] & & {\bullet}
}
\end{align*}
\caption{Quivers with relations of the Auslander algebras of the
diagram $A_3$ with linear (left) and bipartite (right) orientations.}
\label{fig:ausA3}
\end{figure}

A nicer picture is obtained if, instead of considering $\Aus(kQ)$, one
considers the \emph{stable} Auslander algebra, which in our case can be
defined as $\sAus(kQ) = \End(M)$, where $M$ is the sum of all
indecomposable non-projective $kQ$-modules. Looking again at
Figure~\ref{fig:ausA3}, we see that while the Auslander algebras are
not derived equivalent, when restricting to the stable part (the black
dots) we get two derived equivalent algebras of type $A_3$.

This holds in general, as shown by the theorem below, which follows
from recent results of~\cite{IyamaOppermann09b}, see
also~\cite[Prop.~A.2]{GLS07}. An alternative approach will be presented
in the forthcoming paper~\cite{Ladkani09b}.
\begin{thm*}
Let $Q$ and $Q'$ be two orientations of a Dynkin diagram. Then $\sAus(kQ)$
and $\sAus(kQ')$ are derived equivalent.
\end{thm*}

Using this result, we deduce the following.

\begin{cor} \label{c:sAus}
Let $Q$ be a Dynkin quiver whose underlying graph is not $A_{2n}$.
Then the stable Auslander algebra of $kQ$ is derived equivalent to
an incidence algebra of a poset as indicated in the following table.
\begin{center}
\begin{tabular}{ccr@{$=$}l}
& Derived type of & \multicolumn{2}{c}{} \\
Diagram & stable Auslander algebra & 
\multicolumn{2}{c}{CY-dimension} \\
\hline $A_{2n+1}$ & $A_{2n+1} \times A_n$ &
$\frac{4n-2}{2n+2}$ & $\frac{2n}{2n+2} + \frac{n-1}{n+1}$
\vspace{3pt} \\
$D_{2n}$ & $D_{2n} \times A_{2n-2}$ & $\frac{4n-5}{2n-1}$ &
$\frac{2n-2}{2n-1} + \frac{2n-3}{2n-1}$
\vspace{3pt} \\
$D_{2n+1}$ & $D_{2n+1} \times A_{2n-1}$ & $\frac{8n-6}{4n}$
& $\frac{4n-2}{4n} + \frac{2n-2}{2n}$
\vspace{3pt} \\
$E_6$ & $E_6 \times A_5$ & $\frac{18}{12}$ & $\frac{10}{12} +
\frac{4}{6}$
\vspace{3pt} \\
$E_7$ & $E_7 \times A_8$ & $\frac{15}{9}$ & $\frac{8}{9} + \frac{7}{9}$
\vspace{3pt} \\
$E_8$ & $E_8 \times A_{14}$ & $\frac{27}{15}$ & $\frac{14}{15} +
\frac{13}{15}$
\end{tabular}
\end{center}
In particular, all these stable Auslander algebras are fractionally
Calabi-Yau.
\end{cor}

Finally, we obtain the following connection between some ``triangles''
and ``rectangles''. An example (for $n=3$) is shown in
Figure~\ref{fig:trirect}.

\begin{figure}[h]
\begin{align*}
\xymatrix@R=1pc@C=0.5pc{
& & & & & {\bullet} \ar[dr] \\
& & & & {\bullet} \ar[ur] \ar[dr] && {\bullet} \ar@{.}[ll] \ar[dr] \\
& & & {\bullet} \ar[ur] \ar[dr] && {\bullet} \ar@{.}[ll] \ar[ur] \ar[dr]
&& {\bullet} \ar@{.}[ll] \ar[dr] \\
& & {\bullet} \ar[ur] \ar[dr] && {\bullet} \ar@{.}[ll] \ar[ur] \ar[dr]
&& {\bullet} \ar@{.}[ll] \ar[ur] \ar[dr] && {\bullet} \ar@{.}[ll] \ar[dr] \\
& {\bullet} \ar[ur] \ar[dr] && {\bullet} \ar@{.}[ll] \ar[ur] \ar[dr]
&& {\bullet} \ar@{.}[ll] \ar[ur] \ar[dr] && {\bullet} \ar@{.}[ll] \ar[ur] \ar[dr]
&& {\bullet} \ar@{.}[ll] \ar[dr] \\
{\bullet} \ar[ur] && {\bullet} \ar@{.}[ll] \ar[ur]
&& {\bullet} \ar@{.}[ll] \ar[ur] && {\bullet} \ar@{.}[ll] \ar[ur]
&& {\bullet} \ar@{.}[ll] \ar[ur] && {\bullet} \ar@{.}[ll]
}
& &
\xymatrix@R=1pc@C=0.5pc{
& & & {\bullet} \ar[dr] && {\bullet} \ar@{.}[ll] \ar[dr]
&& {\bullet} \ar@{.}[ll] \\
& & {\bullet} \ar[ur] \ar[dr] && {\bullet} \ar@{.}[ll] \ar[ur] \ar[dr]
&& {\bullet} \ar@{.}[ll] \ar[ur] \\
& {\bullet} \ar[ur] \ar[dr] && {\bullet} \ar@{.}[ll] \ar[ur] \ar[dr]
&& {\bullet} \ar@{.}[ll] \ar[ur] \\
{\bullet} \ar[ur] \ar[dr] && {\bullet} \ar@{.}[ll] \ar[ur] \ar[dr]
&& {\bullet} \ar@{.}[ll] \ar[ur] \ar[dr] \\
& {\bullet} \ar[ur] \ar[dr] && {\bullet} \ar@{.}[ll] \ar[ur] \ar[dr]
&& {\bullet} \ar@{.}[ll] \ar[dr] \\
& & {\bullet} \ar[ur] \ar[dr] && {\bullet} \ar@{.}[ll] \ar[ur] \ar[dr]
&& {\bullet} \ar@{.}[ll] \ar[dr] \\
& & & {\bullet} \ar[ur] && {\bullet} \ar@{.}[ll] \ar[ur]
&& {\bullet} \ar@{.}[ll]
}
\\ \\
\xymatrix@R=1pc@C=0.5pc{
{\bullet} \ar[dr] && {\bullet} \ar[ll] \ar[dr] &&
{\bullet} \ar[ll] \ar[dr] \\
& {\bullet} \ar@{.}[ur] \ar[dr]
&& {\bullet} \ar[ll] \ar@{.}[ur] \ar[dr] &&
{\bullet} \ar[ll] \ar[dr] \\
& & {\bullet} \ar@{.}[ur] \ar[dr]
&& {\bullet} \ar[ll] \ar@{.}[ur] \ar[dr] &&
{\bullet} \ar[ll] \ar[dr] \\
& & & {\bullet} \ar@{.}[ur] \ar[dr]
&& {\bullet} \ar[ll] \ar@{.}[ur] \ar[dr] &&
{\bullet} \ar[ll] \ar[dr] \\
& & & & {\bullet} \ar@{.}[ur] \ar[dr]
&& {\bullet} \ar[ll] \ar@{.}[ur] \ar[dr] &&
{\bullet} \ar[ll] \ar[dr] \\
& & & & & {\bullet} \ar@{.}[ur] \ar[dr]
&& {\bullet} \ar[ll] \ar@{.}[ur] \ar[dr] &&
{\bullet} \ar[ll] \ar[dr] \\
& & & & & & {\bullet} \ar@{.}[ur] && {\bullet} \ar[ll] \ar@{.}[ur]
&& {\bullet} \ar[ll]
}
& &
\xymatrix@R=1pc@C=0.5pc{
& & & {\bullet} \ar@{.}[dr] && {\bullet} \ar[ll] \ar@{.}[dr]
&& {\bullet} \ar[ll] \\
& & {\bullet} \ar[ur] \ar@{.}[dr] && {\bullet} \ar[ll] \ar[ur] \ar@{.}[dr]
&& {\bullet} \ar[ll] \ar[ur] \\
& {\bullet} \ar[ur] \ar@{.}[dr] && {\bullet} \ar[ll] \ar[ur] \ar@{.}[dr]
&& {\bullet} \ar[ll] \ar[ur] \\
{\bullet} \ar[ur] \ar[dr] && {\bullet} \ar[ll] \ar[ur] \ar[dr]
&& {\bullet} \ar[ll] \ar[ur] \ar[dr] \\
& {\bullet} \ar@{.}[ur] \ar[dr] && {\bullet} \ar[ll] \ar@{.}[ur] \ar[dr]
&& {\bullet} \ar[ll] \ar[dr] \\
& & {\bullet} \ar@{.}[ur] \ar[dr] && {\bullet} \ar[ll] \ar@{.}[ur] \ar[dr]
&& {\bullet} \ar[ll] \ar[dr] \\
& & & {\bullet} \ar@{.}[ur] && {\bullet} \ar[ll] \ar@{.}[ur]
&& {\bullet} \ar[ll]
}
\end{align*}
\caption{Quivers with relations of four derived equivalent algebras.
Starting at the upper left and going clockwise: the triangle
$\sAus(kA_7)$ with linear orientation on $A_7$; $\sAus(kA_7)$ with a
symmetric orientation on $A_7$; $kA_7 \ten_k kA_3$ with a symmetric
orientation (on $A_7$); and the rectangle $kA_7 \ten_k kA_3$ (with
linear orientation).} \label{fig:trirect}
\end{figure}
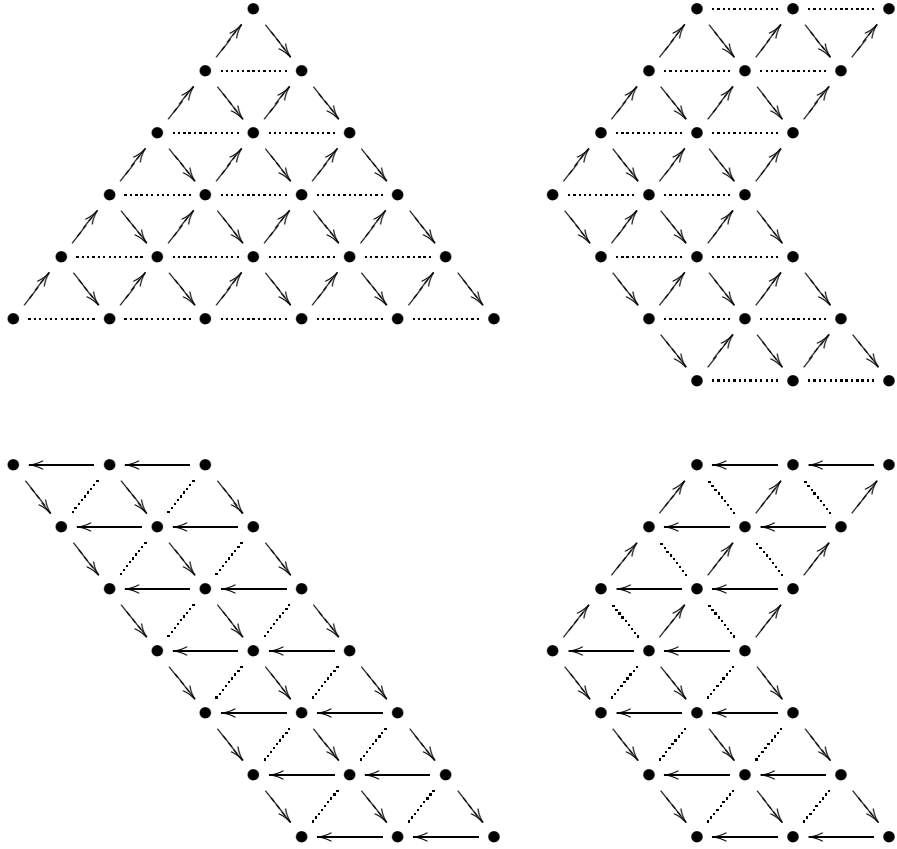

\begin{cor} \label{c:rectriang}
The Auslander algebra corresponding to the linear orientation on $A_{2n}$
is derived equivalent to $kA_{2n+1} \ten_k kA_n$.
\end{cor}

\begin{remark}
It turns out that the fractions appearing in Corollaries~\ref{c:Aus}
and~\ref{c:sAus} as the CY-dimensions of the corresponding (stable)
Auslander algebras are the best possible, i.e.\ common factors
cannot be canceled.
This can be seen for example by considering the Coxeter polynomials of
the two Dynkin quivers comprising the derived type
(see~\cite[\S1.1]{LenzingdelaPena08}) and examining the orders of
products of pairs of roots, one from each polynomial.
\end{remark}

\section{Proof of Theorem~\protect{\ref{t:tensor}}}
\label{sec:prooftensor}

In this section we prove Theorem~\ref{t:tensor}, using calculations
involving tensor products of complexes. For the convenience of the
reader, we give a rather detailed account, see
also~\cite[Chapter~6]{KoenigZimmermann98} for a similar treatment.

Let $k$ be a commutative ring and let $A$ be a $k$-algebra. Our
conventions are that modules are right modules and $\ten$ (without
subscript) will always mean tensor product over $k$. Denote by $\Mod A$
the category of all (right) $A$-modules, and by $\Proj A$ and $\proj A$
the full subcategories of projective and finitely-generated projective
modules, respectively.

Let $B$ be another $k$-algebra. If $M$ is an $A$-module and $N$ is a
$B$-module, then $M \ten N$ is a module over the $k$-algebra $A \ten
B$.

\begin{lemma} \label{l:PQ}
Let $P \in \proj A$ and $Q \in \proj B$.
Then
\begin{enumerate}
\renewcommand{\theenumi}{\alph{enumi}}
\item
$P \ten Q \in \proj(A \ten B)$.
\item
For any $M \in \Mod A$ and $N \in \Mod B$, we have a natural
isomorphism
\[
\Hom_A(P, M) \ten \Hom_B(Q, N) \to \Hom_{A \ten B}(P \ten Q, M \ten N) .
\]
\end{enumerate}
\end{lemma}
\begin{proof}
Both claims are true for $P=A$ and $Q=B$, and pass to finite sums and direct
summands.
\end{proof}

For an additive category $\cA$, denote by $\cC(\cA)$ the category of
complexes over $\cA$, by $\cC^{-}(\cA)$ its full
subcategory consisting of complexes bounded above and by $\cC^b(\cA)$ the
subcategory of bounded complexes. We abbreviate $\cC(\Mod A)$,
$\cC^-(\Mod A)$, $\cC^b(\Mod A)$ by $\cC(A)$, $\cC^-(A)$, $\cC^b(A)$ and
denote by $\dA$, $\cD^-(A)$, $\cD^b(A)$ the corresponding derived categories.

Recall~\cite[\S V.9]{MacLane63} that if $K=(K^p)$ and $L=(L^q)$
are complexes of $k$-modules, their
\emph{tensor product} $K \ten L$ is the complex whose terms are
\[
(K \ten L)^n = \bigoplus_{p+q=n} K^p \ten L^q
\]
with the differential defined on each piece $K^p \ten L^q$ by
\[
d_{K \ten L}(x \ten y) = d_K(x) \ten y + (-1)^p x \ten d_L(y) .
\]

If $X \in \cC(A)$ and $Y \in \cC(B)$, then $X \ten Y \in
\cC(A \ten B)$. When $Y \in \cC^-(B)$ and the terms of $Y$ are
projective as $k$-modules, the functor
$- \ten Y : \cC(A) \to \cC(A \ten B)$ is exact, hence induces
a triangulated functor from $\dA$ to $\dAB$,
which restricts to $\cD^-(A) \to \cD^-(A \ten B)$.

If $T$ and $X$ are complexes of $A$-modules, the complex
$\Hom^{\bullet}_A(T, X)$ is defined by
\begin{equation} \label{e:totHom}
\Hom^{\bullet}_A(T, X)^n = \prod_{p \in \bZ} \Hom_A(T^p, X^{n+p}),
\end{equation}
with the differential taking $f=(f^p:T^p \to X^{n+p})_{p \in \bZ}$ to
$df$, whose $p$-th component is
\[
(df)^p = d_X^{n+p} f^p - (-1)^n f^{p+1} d_T^p .
\]
We have $\h^r \bigl(\Hom^{\bullet}_A(T,X)\bigr) \simeq \Hom_{\cK(A)}(T, X[r])$
where $\cK(A)$ is the homotopy category of complexes.

\begin{lemma} \label{l:PQcomplex}
Let $P \in \cC^b(\proj A)$, $Q \in \cC^b(\proj B)$
and $X \in \cC^-(\Mod A)$, $Y \in \cC^-(\Mod B)$.
Then the natural map
\begin{equation} \label{e:PQcomplex}
\Hom^{\bullet}_A(P, X) \ten \Hom^{\bullet}_B(Q, Y) \to
\Hom^{\bullet}_{A \ten B}(P \ten Q, X \ten Y)
\end{equation}
is an isomorphism.
\end{lemma}
\begin{proof}
Since $P$, $Q$ are bounded and $X$, $Y$ are bounded above, the
complexes in both sides of~\eqref{e:PQcomplex} are bounded above, and
the product in~\eqref{e:totHom} can be replaced by a direct sum.

The $n$-th term of the complex in the left hand side
of~\eqref{e:PQcomplex} is the sum
\[
\bigoplus_{i, j, p, q=n-p} \Hom_A(P^i, X^{p+i}) \ten \Hom_B(Q^j,
Y^{q+j})
\]
with the differential
\[
d(f \ten g) = d_X f \ten g - (-1)^p f d_P \ten g + (-1)^p f \ten d_Y g
- (-1)^{p+q} f \ten g d_Q
\]
for $f \in \Hom_A(P^i, X^{p+i})$, $g \in \Hom_B(Q^j, Y^{q+j})$.

Similarly, the $n$-th term of the complex in the right hand side
of~\eqref{e:PQcomplex} is the sum
\[
\bigoplus_{i, j, p, q=n-p} \Hom_A(P^i \ten Q^j, X^{p+i} \ten Y^{q+j})
\]
with the differential
\begin{align*}
d(f \ten g) &= d_{X \ten Y}(f \ten g) - (-1)^{p+q}(f \ten g)d_{P \ten
Q} \\
&= d_X f \ten g +(-1)^{p+i} f \ten d_Y g - (-1)^{p+q} \bigl( f d_P \ten
g + (-1)^i f \ten g d_Q \bigr)
\end{align*}
for $f \ten g \in \Hom_{A \ten B}(P^i \ten Q^j, X^{p+i} \ten Y^{q+j})$.

The two complexes are now isomorphic by using the isomorphism of
Lemma~\ref{l:PQ} and mapping $f \ten g \mapsto (-1)^{iq} f \ten g$,
where $f \in \Hom_A(P^i, X^{p+i})$ and $g \in \Hom_B(Q^j, Y^{q+j})$.
\end{proof}

\begin{lemma}
\label{l:KLqis}
Let $K \in \cC^-(\Mod k)$, and let $g : L' \to L$ be a map of complexes in
$\cC^-(\Proj k)$ which is a quasi-isomorphism. Then
\[
1_K \ten g : K \ten L' \to K \ten L
\]
is also a quasi-isomorphism.
\end{lemma}
\begin{proof}
Choose a projective resolution of $K$, that is, a quasi-isomorphism
$f : K' \to K$ with $K' \in \cC^-(\Proj k)$. Then in the following commutative
square
\[
\xymatrix{
{K' \ten L'} \ar[r]^{1 \ten g} \ar[d]^{f \ten 1} &
{K' \ten L} \ar[d]^{f \ten 1} \\
{K \ten L'} \ar[r]^{1 \ten g} & {K \ten L}
}
\]
the vertical maps and the top horizontal map are quasi-isomorphisms. It
follows that the bottom horizontal map is also a quasi-isomorphism.
\end{proof}

\begin{cor} \label{c:qisTQU}
Let $X, X' \in \cC^-(\Mod A)$ and let $Y' \to Y$ be a map of complexes
in $\cC^-(\Mod B)$ which is a quasi-isomorphism. Assume that the terms of
$Y$ and $Y'$ are projective as $k$-modules. If $X \simeq X'$ in $\dA$,
then $X \ten Y \simeq X' \ten Y'$ in $\dAB$.
\end{cor}
\begin{proof}
We have $X' \ten Y' \simeq X \ten Y' \xrightarrow{\simeq} X \ten Y$,
where the first isomorphism follows from the exactness of $- \ten Y'$
and the second from Lemma~\ref{l:KLqis}.
\end{proof}

The following lemma is a variation on the K\"{u}nneth
formula~\cite[\S V.10]{MacLane63}, relating the cohomology of a tensor product
of complexes with the tensor product of the cohomologies.

\begin{lemma} \label{l:HKL}
Let $K \in \cC^-(\Mod k)$ and $L \in \cC(\Proj k)$.
Assume that the cohomology of $L$ is concentrated in degree $0$ and that
$\h^0(L)$ is projective. Then the natural map
\[
\h^{\bullet}(K) \ten \h^{\bullet}(L) \to \h^{\bullet}(K \ten L)
\]
is an isomorphism.
\end{lemma}
\begin{proof}
Let $d_L$ denote the differential of $L$. Then from the exact sequence
\[
\ker d^0_L \to \h^0(L) \to 0
\]
and the assumption that $\h^0(L)$ is projective, we deduce the existence
of a map $s : \h^0(L) \to \ker d^0_L \hookrightarrow L^0$,
such that the map of complexes
\[
\xymatrix{
L' : & {\dots} \ar[r] & 0 \ar[d] \ar[r] & \h^0(L) \ar[d]^s \ar[r] &
0 \ar[d] \ar[r] & \dots \\
L : & {\dots} \ar[r] & L^{-1} \ar[r] & L^0 \ar[r] & L^1 \ar[r] & \dots
}
\]
is a quasi-isomorphism.

By Lemma~\ref{l:KLqis}, the induced map $K \ten L' \to K \ten L$ is
a quasi-isomorphism, hence
\begin{align*}
\h^{\bullet}(K) \ten \h^{\bullet}(L) &=
\h^{\bullet}(K) \ten \h^{0}(L) \simeq
\h^{\bullet}(K \ten \h^{0}(L)) \\
&= \h^{\bullet}(K \ten L') \simeq \h^{\bullet}(K \ten L)
\end{align*}
where the exactness of $- \ten \h^0(L)$ implies the isomorphism at the
first row.
\end{proof}

\begin{lemma} \label{l:P}
Assume that $B$ is projective as a $k$-module and let $Q \in \proj B$.
Then:
\begin{enumerate}
\renewcommand{\theenumi}{\alph{enumi}}
\item
\label{i:Q} $Q \in \Proj k$.
\item
\label{i:QN} $\Hom_B(Q, N) \in \Proj k$ for any $N \in \Mod B$ which is
projective as a $k$-module.
\end{enumerate}
\end{lemma}
\begin{proof}
Both claims are true for $Q=B$. Now the first passes to (arbitrary)
sums and direct summands, while the second passes to finite sums and
direct summands.
\end{proof}

For a complex $T \in \dA$, let $\langle \add T \rangle$ denote the
smallest triangulated subcategory of $\dA$ closed under isomorphisms
and taking direct summands. We say that $T$ \emph{generates}
$\langle \add T \rangle$. In particular, $A$ generates $\per A$,
the full subcategory of \emph{perfect complexes} in $\dA$,
which consists of all the complexes that are isomorphic
(in $\dA$) to bounded ones with finitely generated projective terms.

\begin{prop} \label{p:TUtensor}
Assume that $B$ is projective as a $k$-module. Let $T \in \per A$, $U \in
\per B$, $T' \in \cD^-(A)$ and $U' \in \cD^-(B)$. Assume that:
\begin{enumerate}
\renewcommand{\theenumi}{\roman{enumi}}
\item
\label{i:HomUUzero}
$\Hom_{\dB}(U,U'[r]) = 0$ for all $r \neq 0$,
\item
\label{i:UUproj}
The terms of $U$ and $U'$ are projective as $k$-modules,
\item
\label{i:HomUUproj}
$\Hom_{\dB}(U, U')$ is projective as $k$-module,
\end{enumerate}

Then $T \ten U \in \per(A \ten B)$ and for any $r \in \bZ$,
\[
\Hom_{\dAB}\bigl(T \ten U, (T' \ten U')[r]\bigr) \simeq
\Hom_{\dA}(T, T'[r]) \ten \Hom_{\dB}(U, U')
\]
\end{prop}
\begin{proof}
There exist $P \in \cC^b(\proj A)$ which is isomorphic to $T$ in $\dA$
and $Q \in \cC^b(\proj B)$ which is isomorphic to $U$ in $\dB$. By
Lemma~\ref{l:PQ}, $P \ten Q \in \cC^b(\proj(A \ten B))$. Since the
isomorphism between $Q$ and $U$ in $\dB$ can be represented by a
``roof'' of quasi-isomorphisms in $\cC(B)$
\[
\xymatrix@=1pc{
& \widetilde{Q} \ar[dl] \ar[dr] \\
{Q} & & {U}
}
\]
with $\widetilde{Q} \in \cC^{-}(\Proj B)$, we deduce from
Corollary~\ref{c:qisTQU} that $P \ten Q$ is isomorphic to $T \ten U$ in
$\dAB$, as $Q$, $\widetilde{Q}$ and $U$ are in $\cC^-(\Proj k)$ by
Lemma~\ref{l:P}(\ref{i:Q}) and the assumption~(\ref{i:UUproj}).

Therefore $T \ten U$ is perfect and
\begin{align*}
\Hom_{\dAB}\bigl(T \ten U, (T' \ten U')[r]\bigr) =
\Hom_{\dAB}\bigl(P \ten Q, (T' \ten U')[r]\bigr) = \\
\Hom_{\cK(A \ten B)}\bigl(P \ten Q, (T' \ten U')[r]\bigr) =
\h^r \bigl(\Hom^{\bullet}_{A \ten B}(P \ten Q, T' \ten U')\bigr)
\end{align*}
where $\cK(A \ten B)$ is the homotopy category of $\cC(A \ten B)$.

Now, by Lemma~\ref{l:PQcomplex},
\[
\Hom^{\bullet}_{A \ten B}(P \ten Q, T' \ten U') \simeq
\Hom^{\bullet}_A(P, T') \ten \Hom^{\bullet}_B(Q, U') .
\]

Consider the complex of $k$-modules $\Hom^{\bullet}_B(Q, U')$. By
Lemma~\ref{l:P}(\ref{i:QN}), its terms are projective. Moreover, as
\[
\h^r \bigl(\Hom^{\bullet}_B(Q, U') \bigr) \simeq \Hom_{\dB}(U, U'[r]),
\]
hypotheses~(\ref{i:HomUUzero}) and~(\ref{i:HomUUproj}) imply that
its cohomology is concentrated in degree zero and the zeroth cohomology is
projective as $k$-module. Thus, by Lemma~\ref{l:HKL},
\begin{align*}
\h^r \bigl(\Hom^{\bullet}_A(P, T') \ten \Hom^{\bullet}_B(Q, U') \bigr)
&\simeq
\h^r \bigl(\Hom^{\bullet}_A(P, T') \bigr) \ten
\h^0 \bigl(\Hom^{\bullet}_B(Q, U') \bigr) \\
&\simeq \Hom_{\dA}(T, T'[r]) \ten \Hom_{\dB}(U, U')
\end{align*}
and the claim follows.
\end{proof}

\begin{lemma} \label{l:TUgen}
Assume that $B$ is projective as a $k$-module.
Let $U = U_1 \oplus \dots \oplus U_n$ be a complex in $\cD^-(B)$ that
generates $\per B$ whose terms are projective as $k$-modules, and let
$T_1, \dots, T_n$ be complexes in $\per A$ such that each $T_i$
generates $\per A$.
Then the complex
$(T_1 \ten U_1) \oplus (T_2 \ten U_2) \oplus \dots \oplus
(T_n \ten U_n)$
generates $\per (A \ten B)$.
\end{lemma}
\begin{proof}
Each of the summands $T_i \ten U_i$ lies in $\per (A \ten B)$ by
Proposition~\ref{p:TUtensor}. Moreover, by the argument at the
beginning of the proof of that proposition, we may (and will) assume
that $T_i$ and $U_i$ are bounded complexes of finitely generated
projective modules.

Now $A \in \langle \add T_i \rangle$,
since $T_i$ generates $\per A$, hence $A \ten U_i \in \langle \add
(T_i \ten U_i) \rangle$, as $- \ten U_i$ is exact and commutes with
taking direct summands.

Therefore, it is enough to show that in $\cK^b(\proj (A \ten B))$,
\[
A \ten B \in \langle \add (A \ten U_1), \dots, \add (A \ten U_n)
\rangle.
\]
This follows from the facts that $B \in \langle \add (U_1 \oplus \dots
\oplus U_n) \rangle$, as $U$ generates $\per B$, and $A \ten - :
\cK^b(\proj B) \to \cK^b(\proj (A \ten B))$ is an exact functor.
\end{proof}

Now we have all the ingredients to complete the proof of
Theorem~\ref{t:tensor}.

\begin{proof}[Proof of Theorem~\ref{t:tensor}]
Since the terms of $U_i$ are projective as $k$-modules and $T_i \in \per A$,
we may assume that $T_i \in \cD^-(A)$ by replacing it with a suitable
quasi-isomorphic complex.

The claim that the complex $V = \bigoplus_{i=1}^n (T_i \ten U_i)$ is
perfect, exceptional and that its endomorphism algebra is isomorphic to the
one in~\eqref{e:endo}, is a direct consequence of Proposition~\ref{p:TUtensor}.
Finally, by Lemma~\ref{l:TUgen}, $V$ generates $\per (A \ten B)$.
\end{proof}

\section{Proof of Theorems~\protect{\ref{t:ade}}, \protect{\ref{t:endo}} and
their applications}
\label{sec:proofother}

\subsection{The path algebra of $A_n$}

We quickly review some relevant facts concerning the quiver $A_n$. Let
$R$ be a ring, and denote by $R A_n$ the ring of upper-triangular
$n$-by-$n$ matrices with entries in $R$. More abstractly, it is the
free $R$-module on the basis $\{e_{ij}\}_{1 \leq i \leq j \leq n}$,
with the multiplication given by the rules $r e_{ij} = e_{ij} r$ for $r
\in R$ and $e_{ij} e_{pq}$ equals $e_{iq}$ if $p=j$ and zero otherwise.
When $R$ is commutative, $RA_n$ is known as the \emph{path algebra} of
$A_n$ over $R$, and also as the incidence algebra over $R$ of the
linearly ordered poset on $\{1,2,\dots,n\}$.

The category of (right) $RA_n$-modules is equivalent to the category of
diagrams of $R$-modules of the shape $A_n$. In other words, a module
$M$ over $RA_n$ can be described as a diagram
\[
M(1) \to M(2) \to \dots \to M(n)
\]
of $R$-modules, and a morphism of $RA_n$-modules is just a morphism of
diagrams. Namely, $M(i) = M e_{ii}$ and the map $M(i) \to M(i+1)$ is
given by the right action of $e_{i,i+1}$.

Consider the following $RA_n$-modules, for $1 \leq i \leq n$,
\[
\xymatrix@R=0.5pc@C=1.5pc{
P_i: & 0 \ar[r] & \dots \ar[r] & 0 \ar[r] & R \ar[r]^{1_R} &
R \ar[r]^{1_R} & \dots \ar[r]^{1_R} & R
\\
S_i: & 0 \ar[r] & \dots \ar[r] & 0 \ar[r] & R \ar[r] &
0 \ar[r] & \dots \ar[r] & 0
\\
I_i: & R \ar[r]^{1_R} & \dots \ar[r]^{1_R} & R \ar[r]^{1_R} & R \ar[r] &
0 \ar[r] & \dots \ar[r] & 0
}
\]
where $S_i(i) = R$. These modules are free as $R$-modules, and from the
adjunction
\begin{equation} \label{e:adjPiM}
\Hom_{RA_n}(P_i, M) \simeq \Hom_R(R, M(i)) = M(i)
\end{equation}
we see that $P_i$ are projective $RA_n$-modules. In fact, $RA_n =
\bigoplus_{i=1}^n P_i$.
Note that when $R$ is a field, $P_i$, $I_i$ and $S_i$ are the
indecomposable projective, injective and simple corresponding to the vertex
$i$.

We have $S_n = P_n$, $S_1 = I_1$, $I_n = P_1$, and there are short exact
sequences
\begin{align}
\notag
0 \to P_{i+1} \to P_i \to S_i \to 0 && 1 \leq i < n, \\
\label{e:PSI}
0 \to P_{i+1} \to P_1 \to I_i \to 0 && 1 \leq i < n, \\
\notag
0 \to S_i \to I_i \to I_{i-1} \to 0 && 1 < i \leq n .
\end{align}

\begin{prop}{\ } \label{p:RAntilt}
\begin{enumerate}
\renewcommand{\theenumi}{\alph{enumi}}
\item
\label{i:RAntiltP} $P_1 \oplus \dots \oplus P_n$ is a tilting complex
in $\cD(RA_n)$ with endomorphism ring $RA_n$.

\item
\label{i:RAntiltI} $I_1 \oplus \dots \oplus I_n$ is a tilting complex
in $\cD(RA_n)$ with endomorphism ring $RA_n$. In particular,
\[
\Hom_{\cD(RA_n)}(I_i, I_j) \simeq \Hom_{\cD(RA_n)}(P_i, P_j) =
\begin{cases}
R & \text{if $j \leq i$,} \\ 0 & \text{otherwise.}
\end{cases}
\]

\item
\label{i:RAntiltS}
$S_1 \oplus S_2[1] \oplus \dots \oplus S_n[n-1]$ is a tilting complex
in $\cD(RA_n)$ and
\[
\Hom_{\cD(RA_n)}(S_i[i-1], S_j[j-1]) =
\begin{cases}
R & \text{if $j-i \in \{0,1\}$,} \\ 0 & \text{otherwise.}
\end{cases}
\]
\end{enumerate}
\end{prop}
\begin{proof}
The first claim is obvious. For the others, note that the short exact
sequences in~\eqref{e:PSI} show that
\[
\per RA_n = \langle \add(P_1 \oplus \dots \oplus P_n) \rangle
= \langle \add(S_1 \oplus \dots \oplus S_n) \rangle
= \langle \add(I_1 \oplus \dots \oplus I_n) \rangle,
\]
hence each of the complexes is prefect and generates $\per RA_n$. The
proof that each complex is exceptional and the computation of the
morphism spaces follow easily from~\eqref{e:PSI} and the
adjunction~\eqref{e:adjPiM}.
\end{proof}

\begin{remark}
It would have been sufficient to have the above discussion for $R=\bZ$.
The results for general $R$ would then follow from Theorem~2.1
of~\cite{Rickard91} (or Theorem~\ref{t:tensor}).
\end{remark}

\subsection{Theorem~\protect{\ref{t:ade}} and its corollaries}

\begin{proof}[Proof of Theorem~\ref{t:ade}]
Let $A=\gL$, $B=\bZ A_n$ and set $U_i = S_i[i-1]$ for $1 \leq i \leq n$.
The result now follows from Proposition~\ref{p:RAntilt}(\ref{i:RAntiltS})
and Theorem~\ref{t:tensor}.
\end{proof}

\begin{proof}[Proof of the equivalence of Theorem~\ref{t:ade} and
Corollary~\ref{c:itertilt}]
We first assume the theorem and prove its corollary.
Let $\gL_i$ and $Q_i$ be as in Corollary~\ref{c:itertilt}.
Set $T_1 = \gL$. For $1 \leq i < n$, let $F_i$ be the equivalence
\[
F_i = - \tenl {_{\gL_{i+1}} {Q_i}_{\gL_i}} : \cD(\gL_{i+1}) \xrightarrow{\sim}
\cD(\gL_i)
\]
taking $\gL_{i+1}$ to $Q_i$, and set $T_{i+1} = F_1 F_2 \cdot \ldots \cdot F_i
(\gL_{i+1})$. Then $T_1, \dots, T_n$ are tilting complexes in $\dgL$,
\begin{align*}
\Hom_{\dgL}(T_i, T_{i+1}[r]) &=
\Hom_{\dgL}(F_1 \dots F_{i-1} \gL_i, F_1 \dots F_{i-1} F_i \gL_{i+1}[r]) \\
&= \Hom_{\cD(\gL_i)}(\gL_i, F_i \gL_{i+1}[r])
=\Hom_{\cD(\gL_i)}(\gL_i, Q_i[r]) \\
&= \begin{cases}
Q_i & \text{if $r=0$,} \\
0 & \text{otherwise,}
\end{cases}
\end{align*}
and $\End_{\dgL}(T_{i+1}) \simeq \End_{\cD(\gL_{i+1})}(\gL_{i+1}) = \gL_{i+1}$.
The result now follows from the theorem.

\smallskip

Conversely, assume Corollary~\ref{c:itertilt} and let $T_1, \dots, T_n$
be as in Theorem~\ref{t:ade}. Set $\gL_i = \End_{\dgL}(T_i)$ and let
$G_i$ be the equivalence
\[
G_i = \RHom(T_i, -) : \dgL \xrightarrow{\sim} \cD(\gL_i)
\]
taking $T_i$ to $\gL_i$. Consider $Q_i = G_i(T_{i+1}) \in \cD(\gL_i)$ for
$1 \leq i < n$. Then $Q_i$ is a tilting complex in $\cD(\gL_i)$ with
endomorphism ring $\gL_{i+1}$, and from
\[
\Hom_{\cD(\gL_i)}(\gL_i, Q_i[r]) =
\Hom_{\cD(\gL_i)}(G_i T_i, G_i T_{i+1}[r]) \simeq
\Hom_{\dgL}(T_i, T_{i+1}[r])
\]
we see that it is isomorphic (in $\cD(\gL_i)$) to a tilting $\gL_i$-module.
So Corollary~\ref{c:itertilt} can be applied to get the required statement.
\end{proof}

\begin{proof}[Proof of Corollary~\ref{c:rectline}]
Let $m, n \geq 1$, set $\gL = k A_m$ and consider the $k$-module $Q =
\Hom_k(\gL, k)$. It is free with basis elements $\{\vphi_{ji}\}_{1 \leq
i \leq j \leq m}$ defined by $\vphi_{ji}(e_{pq}) = \delta_{ip}
\delta_{jq}$, and has a natural $\gL$-bimodule structure given by
$(\lambda \vphi \lambda')(x) = \vphi(\lambda' x \lambda)$ for $\lambda,
\lambda' \in \gL$ and $\vphi \in Q$. Writing this explicitly for the
basis elements, we have
\begin{equation} \label{e:Qe}
\vphi_{ji} e_{pq} = \begin{cases} \vphi_{jq} & \text{if $p=i$ and $q
\leq j$,} \\ 0 & \text{otherwise,}
\end{cases}
\qquad e_{pq} \vphi_{ji} = \begin{cases} \vphi_{pi} & \text{if $p \geq
i$ and $q = j$,} \\ 0 & \text{otherwise.}
\end{cases}
\end{equation}

We can identify $I_i = \bigoplus_{j \leq i} k \vphi_{ij}$, so that as a
right $\gL$-module, $Q = \bigoplus_{i=1}^{n} I_i$ is a tilting module
by Proposition~\ref{p:RAntilt}(\ref{i:RAntiltI}). Moreover, the left
$\gL$-action on $Q$ coincides with that via $\End_{\gL}(Q_{\gL}) \simeq
\gL$ given in that proposition.

The ``rectangle'' algebra is $kA_m \ten kA_n = \gL \ten kA_n$. The
lower triangular matrix algebra in~\eqref{e:itertilt} is isomorphic to
the following upper triangular one.
\[
\begin{pmatrix}
\gL_n & Q_{n-1} & 0 & \ldots & 0 \\
0 & \ddots & \ddots & \ddots & \vdots \\
\vdots & \ddots & \gL_3 & Q_2 & 0 \\
\vdots & \ldots & 0 & \gL_2 & Q_1 \\
0 & \ldots & 0 & 0 & \gL_1
\end{pmatrix} .
\]
Hence, applying Corollary~\ref{c:itertilt} with $Q_i = Q$ and $\gL_i =
\gL$, we get that $\gL \ten kA_n$ is derived equivalent to the matrix
algebra
\[
\begin{pmatrix}
\gL & Q & 0 & \ldots & 0 \\
0 & \ddots & \ddots & \ddots & \vdots \\
\vdots & \ddots & \gL & Q & 0 \\
\vdots & \ldots & 0 & \gL & Q \\
0 & \ldots & 0 & 0 & \gL
\end{pmatrix}
\]
having a basis $\{e_{ij}^{(s)}\}_{1 \leq i \leq j \leq n, 1 \leq s \leq
n} \cup \{\vphi_{ji}^{(s)}\}_{1 \leq i \leq j \leq n, 1 \leq s < n}$
corresponding to the copies of $\gL$ and $Q$, where the multiplication
$e_{ij}^{(s)} e_{pq}^{(s)}$ is the usual one, $e_{pq}^{(s)}
\vphi_{ji}^{(s)}$ and $\vphi_{ji}^{(s)} e_{pq}^{(s+1)}$ are given by
\eqref{e:Qe}, and all other products are zero.

Finally, ``line'' algebra $A(n \cdot m, m+1)$, which is $kA_{nm}$
modulo the ideal generated by all the paths of length $m+1$, has a
$k$-basis $\{\eps_{ij}\}_{1 \leq i \leq j \leq \min(nm, i+m)}$ where
the product $\eps_{ij} \eps_{pq}$ equals $e_{iq}$ if $p=j$ and $q \leq
i+m$, and zero otherwise. One can then directly verify,
using~\eqref{e:Qe}, that the $k$-linear map defined by
\begin{align*}
e_{ij}^{(s)} \mapsto \eps_{(s-1)m+i, (s-1)m+j} && \vphi_{ji}^{(s)}
\mapsto \eps_{(s-1)m+j, sm+i}
\end{align*}
is an isomorphism of algebras, compare similar calculations
in~\cite{Schroer99} in connection with repetitive algebras.
\end{proof}

\begin{proof}[Proof of Corollary~\ref{c:replica}]
This is immediate from Corollary~\ref{c:itertilt},
since when $\gL$ is Gorenstein, $D\gL$ is a tilting module and $\End_{\gL}(D\gL)
\simeq \gL$.
\end{proof}

\subsection{Theorem~\protect{\ref{t:endo}} and its corollaries}

\begin{proof}[Proof of Theorem~\ref{t:endo}]
Let $A=\gL$, $B=\bZ A_n$ and set $U_i = P_{n+1-i}$ for $1 \leq i \leq n$.
The result now follows from Proposition~\ref{p:RAntilt}(\ref{i:RAntiltP})
and Theorem~\ref{t:tensor}.
\end{proof}

\begin{proof}[Proof of Corollary~\ref{c:endoSum}]
Observe that $\End_{\dgL}(T_1 \oplus \dots \oplus T_n)$ is isomorphic to
the matrix ring
\[
\begin{pmatrix}
\End T_1 & \Hom(T_2, T_1) & \Hom(T_3, T_1) & \ldots & \Hom(T_n, T_1) \\
\Hom (T_1, T_2) & \End T_2 & \Hom(T_3, T_2) & \ldots & \Hom(T_n, T_2) \\
\Hom (T_1, T_3) & \Hom (T_2, T_3) & \End T_3 & \ldots & \Hom(T_n, T_3) \\
\vdots & \vdots & \vdots & \ddots & \vdots \\
\Hom (T_1, T_n) & \Hom (T_2, T_n) & \Hom (T_3, T_n) & \dots & \End T_n
\end{pmatrix}
\]
and the condition~(\ref{i:endoSum}) guarantees that all the terms above the
main diagonal vanish. The result now follows from Theorem~\ref{t:endo}.
\end{proof}

\begin{proof}[Proof of Corollary~\ref{c:endoF}]
Let $T_i = F^{e_i} \gL$ for $1 \leq i \leq n$. Then
\[
\Hom_{\dgL}(T_i, T_j[r]) \simeq \Hom_{\dgL}(\gL, F^{e_j-e_i}\gL[r])
\simeq \h^r(F^{e_j-e_i}\gL) ,
\]
so that conditions~(\ref{i:endoFnz}) and~(\ref{i:endoFz}) match the
corresponding ones in Corollary~\ref{c:endoSum}. The result now follows from
that corollary.
\end{proof}

\subsection{Path algebras of quivers and Auslander algebras}

\begin{proof}[Proof of Corollary~\ref{c:preproj}]
Consider the autoequivalence $F=\nu^{-}[1]$ on the derived category
$\cD^b(\modf kQ)$. By our hypotheses, $F^i(kQ) = \tau^{-i} kQ$ for $0
\leq i \leq r$ are concentrated in degree zero, and moreover
$\h^0(F^{-i}(kQ)) = \tau^i kQ = 0$ for $1 \leq i \leq r$. The result
now follows from Corollary~\ref{c:endoF} with $n=r+1$ and the sequence
$0, 1, \dots, r$.
\end{proof}

\begin{proof}[Proof of Corollary~\ref{c:Aus}]
The corresponding path algebras are homogeneous, so that their
Auslander algebras are of the form $\End_{kQ}(kQ \oplus \dots \oplus
\tau^{-r} kQ)$ for some $r$ determined in~\cite{Gabriel80}. Now apply
Corollary~\ref{c:preproj}.
\end{proof}

\begin{proof}[Proof of Corollary~\ref{c:sAus}]
Each of the diagrams in the table admits at least one orientation
whose path algebra is homogeneous. Now the result follows from
Corollary~\ref{c:preproj} and the fact that the derived
equivalence class of the stable Auslander algebra does not depend
on the orientation.
\end{proof}

\begin{proof}[Proof of Corollary~\ref{c:rectriang}]
The Auslander algebra of the linear orientation on $A_{2n}$ is
isomorphic to the stable Auslander algebra of the linear orientation
on $A_{2n+1}$. Now use Corollary~\ref{c:sAus}.
\end{proof}

\bibliographystyle{amsplain}
\bibliography{tensor}

\providecommand{\bysame}{\leavevmode\hbox to3em{\hrulefill}\thinspace}
\providecommand{\MR}{\relax\ifhmode\unskip\space\fi MR }
\providecommand{\MRhref}[2]{%
  \href{http://www.ams.org/mathscinet-getitem?mr=#1}{#2}
}
\providecommand{\href}[2]{#2}
\begin{thebibliography}{10}

\bibitem{ADE08}
\emph{The {ADE} chain}, Workshop at {B}ielefeld, {O}ctober 2008.

\bibitem{HandbookTilting07}
Lidia Angeleri~H{\"u}gel, Dieter Happel, and Henning Krause (eds.),
  \emph{Handbook of tilting theory}, London Mathematical Society Lecture Note
  Series, vol. 332, Cambridge University Press, Cambridge, 2007.

\bibitem{ABST08}
I.~Assem, T.~Br{\"u}stle, R.~Schiffler, and G.~Todorov, \emph{{$m$}-cluster
  categories and {$m$}-replicated algebras}, J. Pure Appl. Algebra \textbf{212}
  (2008), no.~4, 884--901.

\bibitem{Auslander71}
M.~Auslander, \emph{Representation dimension of artin algebras}, Queen Mary
  College, Mathematics Notes, University of London (1971), 1--179.

\bibitem{BondalKapranov89}
A.~I. Bondal and M.~M. Kapranov, \emph{Representable functors, {S}erre
  functors, and reconstructions}, Izv. Akad. Nauk SSSR Ser. Mat. \textbf{53}
  (1989), no.~6, 1183--1205, 1337.

\bibitem{CarterSegalMacdonald95}
Roger Carter, Graeme Segal, and Ian Macdonald, \emph{Lectures on {L}ie groups
  and {L}ie algebras}, London Mathematical Society Student Texts, vol.~32,
  Cambridge University Press, Cambridge, 1995, With a foreword by Martin
  Taylor.

\bibitem{Chen09}
Xiao-Wu Chen, \emph{The stable monomorphism category of a {F}robenius
  category}, Math. Res. Lett. \textbf{18} (2011), no.~1, 125--137.

\bibitem{Gabriel72}
Peter Gabriel, \emph{Unzerlegbare {D}arstellungen. {I}}, Manuscripta Math.
  \textbf{6} (1972), 71--103; correction, ibid. 6 (1972), 309.

\bibitem{Gabriel80}
\bysame, \emph{Auslander-{R}eiten sequences and representation-finite
  algebras}, Representation theory, {I} ({P}roc. {W}orkshop, {C}arleton
  {U}niv., {O}ttawa, {O}nt., 1979), Lecture Notes in Math., vol. 831, Springer,
  Berlin, 1980, pp.~1--71.

\bibitem{GLS08}
Christof Geiss, Bernard Leclerc, and Jan Schr{\"o}er, \emph{Cluster algebra
  structures and semicanoncial bases for unipotent groups}, preprint at
  \texttt{arXiv:math/0703039}.

\bibitem{GLS07}
\bysame, \emph{Auslander algebras and initial seeds for cluster algebras}, J.
  Lond. Math. Soc. (2) \textbf{75} (2007), no.~3, 718--740.

\bibitem{Happel88}
Dieter Happel, \emph{Triangulated categories in the representation theory of
  finite-dimensional algebras}, London Mathematical Society Lecture Note
  Series, vol. 119, Cambridge University Press, Cambridge, 1988.

\bibitem{HappelSeidel10}
Dieter Happel and Uwe Seidel, \emph{Piecewise hereditary {N}akayama algebras},
  Algebr. Represent. Theory \textbf{13} (2010), no.~6, 693--704.

\bibitem{HerschendIyama09}
Martin Herschend and Osamu Iyama, \emph{{$n$}-representation-finite algebras
  and twisted fractionally {C}alabi-{Y}au algebras}, Bull. London Math. Soc.
  \textbf{43} (2011), no.~3, 449--466.

\bibitem{IyamaOppermann09b}
Osamu Iyama and Steffen Oppermann, \emph{Stable categories of
  {$(n+1)$}-preprojective algebras}, preprint at \texttt{arXiv:0912.3412}.

\bibitem{KoenigZimmermann98}
Steffen K{\"o}nig and Alexander Zimmermann, \emph{Derived equivalences for
  group rings}, Lecture Notes in Mathematics, vol. 1685, Springer-Verlag,
  Berlin, 1998, With contributions by Bernhard Keller, Markus Linckelmann,
  Jeremy Rickard and Rapha{\"e}l Rouquier.

\bibitem{Kontsevich98}
Maxim Kontsevich, \emph{Course at the {ENS}}, Paris, 1998.

\bibitem{Lenzing10}
Dirk Kussin, Helmut Lenzing, and Hagen Meltzer, \emph{Nilpotent operators and
  weighted projective lines}, preprint at \texttt{arXiv:1002.3797}.

\bibitem{Ladkani09b}
Sefi Ladkani, \emph{Derived equivalences and mutation sequences for {A}uslander
  algebras}, in preparation.

\bibitem{Ladkani08}
\bysame, \emph{Derived equivalences of triangular matrix rings arising from
  extensions of tilting modules}, Algebr. Represent. Theory \textbf{14} (2011),
  no.~1, 57--74.

\bibitem{LenzingdelaPena08}
Helmut Lenzing and Jos{\'e}~Antonio de~la Pe{\~n}a, \emph{Spectral analysis of
  finite dimensional algebras and singularities}, Trends in representation
  theory of algebras and related topics, EMS Ser. Congr. Rep., Eur. Math. Soc.,
  Z\"urich, 2008, pp.~541--588.

\bibitem{MacLane63}
Saunders Mac~Lane, \emph{Homology}, Die Grundlehren der mathematischen
  Wissenschaften, Bd. 114, Academic Press Inc., Publishers, New York, 1963.

\bibitem{MiyachiYekutieli01}
Jun-ichi Miyachi and Amnon Yekutieli, \emph{Derived {P}icard groups of
  finite-dimensional hereditary algebras}, Compositio Math. \textbf{129}
  (2001), no.~3, 341--368.

\bibitem{Rickard89}
Jeremy Rickard, \emph{Morita theory for derived categories}, J. London Math.
  Soc. (2) \textbf{39} (1989), no.~3, 436--456.

\bibitem{Rickard91}
\bysame, \emph{Derived equivalences as derived functors}, J. London Math. Soc.
  (2) \textbf{43} (1991), no.~1, 37--48.

\bibitem{Ringel80}
Claus~Michael Ringel, \emph{On algorithms for solving vector space problems.
  {II}. {T}ame algebras}, Representation theory, {I} ({P}roc. {W}orkshop,
  {C}arleton {U}niv., {O}ttawa, {O}nt., 1979), Lecture Notes in Math., vol.
  831, Springer, Berlin, 1980, pp.~137--287.

\bibitem{RingelSchmidmeier08}
Claus~Michael Ringel and Markus Schmidmeier, \emph{Invariant subspaces of
  nilpotent linear operators. {I}}, J. Reine Angew. Math. \textbf{614} (2008),
  1--52.

\bibitem{Schroer99}
Jan Schr{\"o}er, \emph{On the quiver with relations of a repetitive algebra},
  Arch. Math. (Basel) \textbf{72} (1999), no.~6, 426--432.

\end{thebibliography}

\end{document}